\documentclass[oneside,article]{memoir}
\usepackage{amsmath}         
\usepackage{amsthm,thmtools,thm-restate} 
\usepackage{enumitem}        
  
\usepackage{tikz-cd}

\usepackage{float} 
\usepackage{caption, subcaption} 

\usepackage[nobysame,alphabetic,initials]{amsrefs} 
\usepackage[hidelinks]{hyperref}                   
\usepackage[capitalize,nosort]{cleveref}           

\usepackage{amssymb}
\usepackage[mathscr]{euscript}
\usepackage{relsize}
\usepackage{stmaryrd}
\usepackage{stackengine}

\usepackage{indentfirst} 
\usepackage{multicol}

\usepackage[inline,nomargin]{fixme} 
  \fxsetup{targetlayout=color}

\usepackage{float} 

\isopage[12]

\setlrmargins{*}{*}{1}
\checkandfixthelayout

\counterwithout{section}{chapter}
\usepackage{appendix}

  \newcommand{\we}{\mathrel\sim}
  
  \newcommand{\op}{{\mathord\mathrm{op}}}

  \newcommand{\ob}{\operatorname{ob}}

  \newcommand{\push}[1]{\cup_{#1}}

  \newcommand{\colim}{\operatorname*{colim}}

  \newcommand{\slice}{\mathbin\downarrow}

  \newcommand{\Sk}{\operatorname{Sk}}

  \newcommand{\bd}{\partial}

  \newcommand{\id}[1][]{\operatorname{id}_{#1}}
  
  \newcommand{\cod}{\operatorname{cod}}

  \newcommand{\simp}[1]{\mathord\Delta^{#1}}

  \newcommand{\nat}{{\mathord\mathbb{N}}}

  \newcommand{\cat}[1]{\mathsf{#1}}

  \newcommand{\ncat}[1]{\mathsf{#1}}

  \newcommand{\sSet}{\ncat{sSet}}

  \newcommand{\from}{\colon}

  \newcommand{\ito}{\hookrightarrow}

  \newcommand{\weto}{\mathrel{\ensurestackMath{\stackon[-2pt]{\xrightarrow{\makebox[.8em]{}}}{\mathsmaller{\mathsmaller\we}}}}}

  \newcommand{\cto}{\rightarrowtail}

  \declaretheorem[style=definition,within=section]{definition}
  \declaretheorem[style=definition,numberlike=definition]{assumption}
  \declaretheorem[style=definition,numberlike=definition]{example}
  
  \declaretheorem[style=definition,numberlike=definition]{notation}
  \declaretheorem[style=definition,numberlike=definition]{remark}

  \declaretheorem[style=plain,numberlike=definition]{corollary}
  \declaretheorem[style=plain,numberlike=definition]{lemma}
  \declaretheorem[style=plain,numberlike=definition]{proposition}
  \declaretheorem[style=plain,numberlike=definition]{theorem}

  \Crefname{nparagraph}{Paragraph}{Paragraphs}

  \declaretheorem[style=plain,numbered=no,name=Theorem]{theorem*}

  \Crefname{corollary}{Corollary}{Corollaries}
  \Crefname{definition}{Definition}{Definitions}
  \Crefname{assumption}{Assumption}{Assumptions}
  \Crefname{lemma}{Lemma}{Lemmas}
  \Crefname{proposition}{Proposition}{Propositions}
  \Crefname{remark}{Remark}{Remarks}
  \Crefname{theorem}{Theorem}{Theorems}

  \newlist{axioms}{enumerate}{1}
  \Crefname{axiomsi}{}{}

  \newenvironment{tikzeq*}
  {
    \begingroup
    \begin{equation*}
    \begin{tikzpicture}[baseline=(current bounding box.center)]
  }
  {
    \end{tikzpicture}
    \end{equation*}
    \endgroup
    \ignorespacesafterend
  }

  \tikzset
  {
    diagram/.style=
    {
      matrix of math nodes,
      column sep={4.3em,between origins},
      row sep={4em,between origins},
      text height=1.5ex,
      text depth=.25ex
    },
    over/.style={preaction={draw=white,-,line width=6pt}},
    every to/.style={font=\footnotesize},
    inj/.style={right hook->},
    surj/.style={-{Latex[open]}},
    cof/.style={>->},
    fib/.style={->>},
  }

  \DeclareFontFamily{U}{mathx}{\hyphenchar\font45}

  \DeclareFontShape{U}{mathx}{m}{n}{
    <5> <6> <7> <8> <9> <10>
    <10.95> <12> <14.4> <17.28> <20.74> <24.88>
    mathx10}{}

  \DeclareSymbolFont{mathx}{U}{mathx}{m}{n}

  \DeclareFontFamily{U}{mathb}{\hyphenchar\font45}

  \DeclareFontShape{U}{mathb}{m}{n}{
    <5> <6> <7> <8> <9> <10>
    <10.95> <12> <14.4> <17.28> <20.74> <24.88>
    mathb10}{}

  \DeclareSymbolFont{mathb}{U}{mathb}{m}{n}

  \DeclareMathAccent{\widebar}{0}{mathx}{"73}

  \DeclareMathSymbol{\Rsh}{\mathrel}{mathb}{"E9}

  \DeclareFontFamily{U}{MnSymbolA}{}

  \DeclareFontShape{U}{MnSymbolA}{m}{n}{
    <-6> MnSymbolA5
    <6-7> MnSymbolA6
    <7-8> MnSymbolA7
    <8-9> MnSymbolA8
    <9-10> MnSymbolA9
    <10-12> MnSymbolA10
    <12-> MnSymbolA12}{}

  \DeclareSymbolFont{MnSyA}{U}{MnSymbolA}{m}{n}

  \DeclareMathSymbol{\twoheaddownarrow}{\mathrel}{MnSyA}{27}

  \newcommand{\MSC}[1]{%
    \let\thempfn\relax
    \footnotetext[0]{2020 Mathematics Subject Classification: #1.}
  }

\newcommand{\cSet}{\mathsf{cSet}}
\newcommand{\ccSet}{\mathsf{ccSet}}
\newcommand{\Set}{\mathsf{Set}}
\newcommand{\ssSet}{\mathsf{ssSet}}

\newcommand{\aSet}{\mathsf{aSet}} 
\newcommand{\rSet}{\mathsf{rSet}} 
\newcommand{\arSet}{\mathsf{arSet}} 
\newcommand{\raSet}{\mathsf{raSet}} 
\newcommand{\aaSet}{\mathsf{aaSet}} 

\newcommand{\diag}[1][\otimes]{d_{#1}} 
\newcommand{\diagonal}{\mathsf{diag}} 
\newcommand{\repr}[1]{\widehat{#1}} 
\newcommand{\extimes}{\mathbin{\underline{\times}}} 
\newcommand{\bi}[1]{\mathbf{#1}} 

\newcommand{\join}{\ast} 

\author{Daniel Carranza \and Krzysztof Kapulkin \and Liang Ze Wong}
\title{Diagonal Lemma for Presheaves on Elegant Reedy Categories}
\date{\today}

\begin{document}

  \maketitle

  \begin{abstract}
  The diagonal lemma asserts that if a map of bisimplicial sets is a levelwise weak equivalence in the Kan--Quillen model structure, then it induces a weak equivalence of the diagonal simplicial sets.
  In this paper, we observe that the standard proof of this fact works in greater generality, namely that of (elegant) Reedy categories.
    \MSC{18N40, 55U35 (primary), 18N45, 18N50 (secondary)}
  \end{abstract}

  \section*{Introduction}

The diagonal lemma is a fundamental result of simplicial homotopy theory \cite[Ch.~IV]{goerss-jardine}.
It states that a map of bisimplicial sets $f \from X \to Y$ that is a levelwise equivalence (i.e., $f_n \from X_n \to Y_n$ is a weak homotopy equivalence for every non-negative integer $n$) induces a weak homotopy equivalence on the diagonal simplicial sets $\diagonal f \from \diagonal X \to \diagonal Y$ (where $(\diagonal X)_m = X_{m, m}$).

The result was independently discovered by Bousfield and Kan \cite[Lems.~XII.4.2--3]{bousfield-kan}, Segal \cite[Prop.~A.1]{segal:cats-and-cohomology}, and Tornehave (cf.~\cite[Rem.~3.14]{latch-thomason-wilson}).
Newer accounts include the seminal text of Goerss and Jardine \cite[Prop.~IV.1.9]{goerss-jardine} and a constructive proof of the Kan--Quillen model structure due to Gambino, Sattler, and Szumi{\l}o \cite[Prop.~2.3.5]{gambino-sattler-szumilo}.

In particular, the proof presented in \cite{gambino-sattler-szumilo} generalizes straightforwardly to other settings in several ways.
The first generalization is abstracting away the notion of weak equivalence in that instead of working with weak homotopy equivalences, one might, for example, consider weak categorical equivalences of the Joyal model structure.
The second generalization has to do with the indexing category --- instead of bisimplicial sets, i.e., functors $\Delta^\op \times \Delta^\op \to \Set$, one might consider more general functors $A^\op \times A^\op \to \Set$ or even $A^\op \times R^\op \to \Set$, where $A$ and $R$ are sufficiently nice categories, for example, (elegant) Reedy.
Moving to the case of $A^\op \times R^\op \to \Set$, one is forced to rethink what it means to be a `diagonal' functor, and the requisite properties of such a functor can be axiomatized.
Putting all these generalizations together, we prove the following version of the usual diagonal lemma.

\begin{theorem*}[cf.~\cref{thm:gen-diag-lem}]
    Suppose $A$ and $R$ are Reedy categories and consider $\Set^{A^\op}$ with a cofibration category structure whose cofibrations are the monomorphisms.
    Let $f \from X \to Y$ be a morphism in $\Set^{A^\op \times R^\op}$ between Reedy cofibrant diagrams such that $f_r \from X_r \to Y_r$ is a weak equivalence in $\Set^{A^\op}$ for all $r \in R$.
    If $\diag$ is a diagonal functor in the sense of \cref{def:diag-fun}, then $\diag f \from \diag X \to \diag Y$ is a weak equivalence.
\end{theorem*}

Examples of applications of this statement abound and we give several in \cref{sec:examples}.
In particular, in subsequent joint work with Lindsey, the first two authors used the above theorem in the context of the Joyal model structure \cite{carranza-kapulkin-lindsey:clf}.
Although in all of these examples, we consider the case $A = R$, the proof is perhaps the cleanest when considered in the more general form stated above.

One can consider further generalizations of the above statement.
One such generalization would be a weakening of the Reedy condition to allow objects to have non-trivial automorphisms.
Another possibility would be to replace $\Set^{A^\op}$ with an arbitrary (Grothendieck) topos.
Although both of these seem plausible, the proof techniques used here do not apply to them, and hence such generalizations could be a subject of future work.

The paper is organized as follows.
In \cref{sec:prelim}, we collect the necessary background on Reedy categories and the homotopical structure of presheaves thereon, which we use as a generalization of the simplex category $\Delta$.
Then in \cref{sec:diag-lem}, we prove the Generalized Diagonal Lemma (\cref{thm:gen-diag-lem}) before giving several examples of interest in \cref{sec:examples}.

\textbf{Acknowledgement.}
We thank the anonymous referee for numerous suggestions that helped improve the quality of the paper, including the generalization from $\Set^{A^\op \times A^\op}$ to $\Set^{A^\op \times R^\op}$, which in turn let us identify the necessary assumptions on $A$ and $R$.

This material is based upon work supported by the National Science Foundation under Grant No.~DMS-1928930 while the first two authors participated in a program hosted by the Mathematical Sciences Research Institute in Berkeley, California, during the 2022--23 academic year.

  \section{Preliminaries} \label{sec:prelim}


In this section, we collect the necessary background on Reedy theory and its extensions, for which an excellent survey is \cite{campion:cubical-sites}.
We begin by recalling the definition of a Reedy category.

\begin{definition}
    A \emph{Reedy category} is a category $R$ with a function $\deg \from \ob R \to \nat$ and two wide subcategories $R_-$ and $R_+$ of $R$ such that:
    \begin{enumerate}
        \item If a non-identity map $r \to s$ is in $R_-$, then $\deg r > \deg s$; if a non-identity map $r \to s$ is in $R_+$, then $\deg r < \deg s$.
        \item For any morphism $\varphi \in R$, there are unique morphisms $\varphi_- \in R_-$ and $\varphi_+ \in R_+$ such that $\varphi = \varphi_+ \varphi_-$.
    \end{enumerate}
\end{definition}
Note that conditions (1) and (2) imply that $R$ has no non-identity isomorphisms.

\begin{example} \label{ex:delta-box-reedy}
  The simplex category $\Delta$ and several variants of the box category $\Box$ (e.g., with or without connections) are Reedy categories (cf.~\cite[Cor.~1.17]{doherty-kapulkin-lindsey-sattler}).
\end{example}

\begin{example} \label{ex:direct-inverse-reedy}
  A category $I$ is \emph{direct} (respectively, \emph{inverse}) if there exists a function $\deg \colon \ob I \to \nat$ such that for every non-identity morphism $i \to j$ in $I$, we have $\deg i < \deg j$ (respectively, $\deg i > \deg j$).
  With these definitions, every direct or inverse category is a Reedy category.
  Moreover, for any Reedy category $R$, the subcategory $R_-$ is inverse and the subcategory $R_+$ is direct.
\end{example}




  



We fix two Reedy categories $A$ and $R$.
We shall use the small letters $a$, $b$, $c$, \ldots and $r$, $s$, $t$, \dots to indicate objects of $A$ and $R$, respectively.
The objects of $A \times R$ are denoted in bold as $\bi{ar} = (a, r)$ or $\bi{bs} = (b, s)$.

The category of presheaves on $A$, i.e., contravariant functors $A^\op \to \Set$ is denoted $\aSet$, and the category of presheaves on $R$ is similarly denoted $\rSet$.
The category of presheaves on $A \times R$ is denoted $\arSet$.

We shall use capital letters $K$, $L$, \ldots to denote presheaves on $A$ or $R$, and letters $X$, $Y$, \ldots to denote presheaves on $A \times R$.
Representable presheaves $A(-, a)$ or $R(-, r)$ represented by $a$ or $r$ are denoted $\repr{a}$ and $\repr{r}$, respectively.
A representable presheaf $A \times R (-, \bi{ar})$ is denoted $\repr{\bi{ar}}$.

For $K \in \aSet$ and $a \in A$, we write $K_a$ for the set $K(a)$.
For $x \in K_a$ and $\varphi \from b \to a$ in $A$, we write $x\varphi \in K_b$ for the application of the function $K(\varphi) \from K_a \to K_b$ to the element $x$.

We will see that a version of the diagonal lemma holds for presheaves on arbitrary Reedy categories (\cref{thm:gen-diag-lem}).
However, an important technical assumption (namely, Reedy cofibrancy) is simplified (\cref{cor:}) by working with \emph{elegant} Reedy categories, a notion due to Bergner and Rezk \cite{bergner-rezk}.
To state it, we need a preliminary definition.
\begin{definition}
  Let $K$ be a presheaf on $R$ and $r$ an object of $K$.
  An element $x \in K_r$ is \emph{degenerate} if there is a non-identity $\sigma \colon r \to s$ in $R_-$ and $y \in K_s$ such that $x = y \sigma$; it is \emph{non-degenerate} if it is not degenerate.
\end{definition}


\begin{definition} \label{lem:ez}
  A Reedy category $R$ is \emph{elegant} if for any presheaf $K \in \rSet$ and any element $x \in K_r$, there is a unique map $\sigma \colon r \to s$ in $R_-$ and a unique non-degenerate element $y \in K_s$ such that $x = y \sigma$.
\end{definition} 

Just like CW complexes, presheaves on Reedy categories have a notion of \emph{skeleta}, which we now define.
For $n \geq -1$, let $R_{\leq n}$ denote the full subcategory of the EZ category $R$ spanned by objects $r$ with $\deg r \leq n$.
(In particular, $R_{\leq -1}$ is the empty category.)
The inclusion $i_n \colon R_{\leq n} \ito R$ induces adjoint triples
\[
  \begin{tikzcd}[column sep=huge]
    \Set^{R_{\leq n}^\op} 
      \ar[r, bend left, "(i_n)_!" above]
      \ar[r, bend right, "(i_n)_*" below]
    & \rSet 
      \ar[l,"i_n^*" description]
    & \aSet^{R_{\leq n}^\op} 
      \ar[r, bend left, "(\id \times i_n)_!" above]
      \ar[r, bend right, "(\id \times i_n)_*" below]
    & \arSet 
      \ar[l,"(\id \times i_n)^*" description]
  \end{tikzcd}
\]

\begin{definition} \leavevmode
\begin{enumerate}
  \item For $n \geq -1$, the \emph{$n$-skeleton} of a presheaf $K \in \rSet$ is the presheaf $\Sk^n K = (i_n)_! i_n^* K$.
  \item For $n \geq -1$, the \emph{$n$-skeleton} of a presheaf $X \in \arSet$ is the presheaf $\Sk^n X = (\id \times i_n)_! (\id \times i_n)^* X$.
\end{enumerate}
\end{definition}
\begin{notation}
For $r \in R$ such that $\deg(r) = n$, we write $\bd \repr{r}$ for the $(n-1)$-skeleton $\Sk^{n-1} \repr{r}$ of $\repr{r}$.
\end{notation}

Recall that every presheaf $K \in \rSet$ is a colimit of representables $K \cong \colim\limits_{(r, x) \in \int_R K} \repr{r}$, 
where $\int_R K$ denotes the category of elements of $K$.
The following result adapts this colimit to give a description of the $n$-skeleton of a presheaf.
\begin{proposition} \label{skeleton-colim-repr}
  Given $n \geq -1$,
  \begin{enumerate}
    \item for any $K \in \rSet$, there is an isomorphism
    \[ \Sk^n K \cong \colim\limits_{\substack{(r, x) \in \int_R K \\ \deg(r) \leq n}} \repr{r} \]
    natural in $K$.
    \item for any $X \in \arSet$, there is an isomorphism
    \[ \Sk^n X \cong \colim\limits_{\substack{(a, r, x) \in \int_{A \times R} X \\ \deg(r) \leq n}} \repr{\bi{ar}} \]
    natural in $X$.
  \end{enumerate}
\end{proposition}
\begin{proof}
  For (1), the category of elements $\int\limits_{R_{\leq n}} \! i_n^* K$ is isomorphic to the full subcategory of $\int\limits_{R}\! K$ consisting of objects $(r \in R, x \in K_r)$ such that $\deg(r) \leq n$.
  By writing $i_n^* K$ as a colimit of representables, this gives an isomorphism 
  \begin{align*} 
    \Sk^n K &= (i_n)_! i_n^* K \\
    &\cong (i_n)_! \big( \colim\limits_{(r, x) \in \int_{R_{\leq n}} \! i_n^* K} \repr{r} \, \big) \\
    &\cong (i_n)_! \big( \colim\limits_{\substack{(r, x) \in \int_R \! K \\ \deg(r) \leq n}} \repr{r} \, \big) \\
    &\cong \colim\limits_{\substack{(r, x) \in \int_R \! K \\ \deg(r) \leq n}} (i_n)_! \repr{r} \\
    &\cong \colim\limits_{\substack{(r, x) \in \int_R \! K \\ \deg(r) \leq n}} \repr{r}.
  \end{align*}

  The argument for (2) is analogous.
\end{proof}
From this description, we deduce the usual colimit sequence of ``skeletal induction''.
\begin{corollary} \label{skeletal-induction}
  \leavevmode
  \begin{enumerate}
    \item For $K \in \rSet$, there is a natural map $\Sk^m K \to \Sk^{n} K$ whenever $m < n$.
    These maps form a diagram 
    \[ \Sk^{-1} K \to \Sk^0 K \to \Sk^1 K \to \cdots \]
    from $\mathbb{N} \to \rSet$, and $K$ is the colimit of this seqeuence.

    \item For $X \in \arSet$, there is a natural map $\Sk^m X \to \Sk^{n} X$ whenever $m < n$.
    These maps form a diagram 
    \[ \Sk^{-1} X \to \Sk^0 X \to \Sk^1 X \to \cdots \]
    from $\mathbb{N} \to \arSet$, and $X$ is the colimit of this seqeuence. \qed
  \end{enumerate}
\end{corollary}
  
  


The $n$-skeleton of a representable presheaf has a convenient explicit description.
\begin{proposition} \label{repr-skeleton-description}
  Let $n \geq -1$.
  \begin{enumerate}
    \item For $r, s \in R$, we have a bijection 
    \[ (\Sk^n \repr{s})_r \cong \{ f \from r \to s \mid \text{$f$ factors through an object of degree $\leq n$} \} \]
    natural in $r$ and $s$.
    In particular, 
    \[ (\bd \repr{s})_r \cong \{ f \from r \to s \mid f_+ \neq \id \}. \]
    \item For $a, b \in A$ and $r, s \in R$, we have a bijection
    \[ (\Sk^n \repr{\bi{bs}})_{a, r} \cong A(a, b) \times \{ f \from r \to s \mid \text{$f$ factors through an object of degree $\leq n$} \} \]
    natural in $a$, $r$, and $\bi{bs}$.
  \end{enumerate}
\end{proposition}
\begin{proof}
  Item (1) is \cite[Lem.~3.17 \& Obs.~3.18]{riehl-verity:reedy-categories}.
  For item (2), we apply \cref{skeleton-colim-repr} to obtain an isomorphism
  \[ (\Sk^n \repr{\bi{bs}}) \cong \colim\limits_{\substack{(c, t, \varphi, \psi) \in \int_{A \times R} \! \repr{\bi{bs}} \\ \deg(t) \leq n}} \repr{\bi{ct}}. \]
  The full subcategory of $\int\limits_{A \times R} \! \repr{\bi{bs}}$ spanned by pairs $(c, t, \varphi, \psi)$ satisfying $\deg(t) \leq n$ is isomorphic to the product category $\int\limits_{A} \repr{b} \times \int\limits_{R_{\leq n}}\! \repr{s}$.
  Thus,
  \begin{align*}
    (\Sk^n \repr{\bi{bs}})_{a, r} &\cong \big( \! \colim\limits_{\substack{(c, \varphi) \in \int_A \repr{b} \\ (t, \psi) \in \int_{R \leq n}\! \repr{s}}} \repr{\bi{ct}} \ \big)_{a, r} \\
    &\cong \colim\limits_{\substack{(c, \varphi) \in \int_A \repr{b} \\ (t, \psi) \in \int_{R \leq n}\! \repr{s}}} (\repr{\bi{ct}})_{a, r} \\
    &\cong \colim\limits_{\substack{(c, \varphi) \in \int_A \repr{b} \\ (t, \psi) \in \int_{R \leq n}\! \repr{s}}} \big( A(a, c) \times R(r, t) \big) \\
    &\cong \colim\limits_{(c, \varphi) \in \int_A \repr{b}} A(a, c) \, \times \, \colim\limits_{(t, \psi) \in \int_{R_{\leq n}}\! \repr{s}} \! R(r, t) \\
    &\cong A(a, b) \times (\Sk^n \repr{s})_r,
  \end{align*}
  from which the result follows by item (1).
\end{proof}
In light of \cref{repr-skeleton-description}, we view $\Sk^n \repr{s}$ and $\Sk^n \repr{\bi{bs}}$ as subobjects of $\repr{s}$ and $\repr{\bi{bs}}$, respectively.

We now extend our considerations from purely category-theoretic notions to include some `homotopical' structure.
As suggested by \cref{ex:delta-box-reedy}, Reedy categories are `nice shape categories' for diagrams taking values in a category with such homotopical structure.
A typical target category could be a model category, but for our purposes a cofibration category (cf.~\cite{brown,radulescu-banu,szumilo}) is sufficient.

\begin{definition}\label{cofib-cat-def}
  A \emph{cofibration category} consists of a category $\cat{C}$ together with two wide subcategories: of \emph{cofibrations}, denoted $\cto$, and of \emph{weak equivalences}, denoted $\weto$, subject to the following conditions (where by an \emph{acyclic cofibration} we mean a morphism that is both a cofibration and a weak equivalence):
  \begin{enumerate}
      \item Weak equivalences satisfy 2-out-of-3.
      \item The category $\cat{C}$ has an initial object $\varnothing$ and for any object $X \in C$, the unique map $\varnothing \to X$ is a cofibration (i.e., all objects are \emph{cofibrant}).
      \item For any object $X \in \cat{C}$, the codiagonal map $X \sqcup X \to X$ can be factored as a cofibration followed by a weak equivalence.
      \item The category $\cat{C}$ admits pushouts along cofibrations.
      Moreover, the pushout of an (acyclic) cofibration is an (acyclic) cofibration.
      \item The category $\cat{C}$ admits (small) coproducts; if $\{ f_i \from X_i \to Y_i \}_{i \in I}$ is a collection of (acyclic) cofibrations then $\coprod f_i \from \coprod X_i \to \coprod Y_i$ is an (acyclic) cofibration.
      \item Given a countable sequence of composable (acyclic) cofibrations $X_1 \to X_1 \to X_2 \to \dots$, the colimit $\colim X_i$ exists and the cone maps $X_i \to \colim X_i$ are (acylic) cofibrations.
  \end{enumerate}
\end{definition}

The following example shows in what way cofibration categories generalize model categories.

\begin{example}
  Given a model category $\cat{M}$, its full subcategory of cofibrant objects forms a cofibration category.
\end{example}

\begin{definition}
  Let $R$ be a Reedy category, $\cat{C}$ a cofibration category, and $X \colon R^\op \to \cat{C}$ a diagram.
  \begin{enumerate}
    \item The \emph{latching category} $\bd(r \slice R_-)$ of $r \in R$ is the full subcategory of the slice category $r \slice R_-$ spanned by all non-identity morphisms $r \to s$ in $R_-$.
    \item The \emph{latching object} of $X$ at $r$ is
    \[ L_r X := \colim \left( \bd(r \slice R_-)^\op \to R^\op \xrightarrow{X} \cat{C}  \right) \]
    \item The diagram $X$ is \emph{Reedy cofibrant} if for every $r \in R$, the induced map $L_r X \to X_r$ is a cofibration in $\cat{C}$.
  \end{enumerate}
\end{definition}
When we say a presheaf $X \in \arSet$ is Reedy cofibrant, we will always mean it is Reedy cofibrant when viewed as a diagram $R^\op \to \aSet$.

An advantage of Reedy cofibrant diagrams is that their colimits are homotopy colimits if the indexing category is an inverse category.
\begin{proposition}[{\cite[Thm.~9.3.5.(1c)]{radulescu-banu}}] \label{reedy-cofib-colim-weq}
  If $R$ is an inverse category (i.e.\ $R^\op$ is a direct category) then a pointwise weak equivalence $f \from X \to Y$ between Reedy cofibrant diagrams $X, Y \from R^\op \to \cat{C}$ induces a weak equivalence
  \[ \colim f \from \colim X \to \colim Y. \tag*{\qedsymbol} \]
\end{proposition}

Recall that a \emph{Cisinski model category} is a model structure on a topos in which cofibrations are the monomorphisms.
The Kan--Quillen model structure on the category of simplicial sets, as well as the Grothendieck model structure on the category of cubical sets, are examples of Cisinski model categories \cite[Prop.~2.1.5 \& Thm.~8.4.38]{cisinski:presheaves}.

We are however interested in working with a weaker structure, namely that of a cofibration category, which motivates the following definition.

\begin{definition} \label{def:cis-cof-cat}
  A \emph{Cisinski cofibration category} is a cofibration category structure on a (Grothendieck) topos in which cofibrations are the monomorphisms.
\end{definition}

\begin{example}
  Since all objects of a Cisinski model category are cofibrant, every Cisinski model structure has an underlying Cisinski cofibration category structure.
\end{example}

Working in the generality of Cisinski cofibration categories means that one can use the Generalized Diagonal Lemma to construct model structures on presheaf categories.
This recovers a useful application of the original diagonal lemma (e.g.\ \cite{gambino-sattler-szumilo}).

The key advantage of working with both elegant Reedy categories and Cisinski cofibration categories is Reedy cofibrancy.
\begin{lemma} \label{elegant-reedy-cofibrant}
  Let $\aSet$ be equipped with a Cisinski cofibration category structure and $R$ be an elegant Reedy category.
  Then any diagram $R^\op \to \aSet$ is Reedy cofibrant.
\end{lemma}
\begin{proof}
  We use \cite[Prop.~3.14]{bergner-rezk}, which shows that if $R$ is an elegant Reedy category and $K \in \rSet$ is a presheaf then the latching map $L_r K \to K$ is a monomorphism for all $b$.
  For a diagram $X \in \arSet$, applying \cite[Prop.~3.14]{bergner-rezk} to $X_a \in \rSet$ gives that $L_r X_a \to X_{a, r}$ is a monomorphism for all $a$ and $r$.
  We have a natural isomorphism
  \begin{align*} 
    L_r (X_a) &= \colim\limits_{f \in \bd(r \slice R_-)} X_{a, {\cod (f)}} \\
    &\cong \left( \colim\limits_{f \in \bd(r \slice R_-)} (X_{\cod (f)}) \right)_a \\
    &= (L_r X)_a.
  \end{align*}
  Thus, $(L_r X)_a \to X_{a, r}$ is a monomorphism for all $a$, hence $L_r X \to X_r$ is a monomorphism.
  As monomorphisms are cofibrations in $\aSet$, the diagram $X \from R^\op \to \aSet$ is Reedy cofibrant.
\end{proof}

  \section{Generalized Diagonal Lemma} \label{sec:diag-lem}

The goal of this section is to state and prove the Generalized Diagonal Lemma, which we do in \cref{thm:gen-diag-lem}.
We begin however by stating our global assumption.

\begin{assumption} \label{A-ez-aset-Cisinski}
  Throughout the remainder of the paper, $A$ and $R$ will be Reedy categories, and $\aSet$ will always be considered with a Cisinski cofibration category structure.
\end{assumption}

As indicated in the introduction, the diagonal lemma can be generalized to other diagonal-like functors.
In order to spell out the requisite properties of such a functor, let us first recall the notion of the \emph{external product} of presheaves --- it is a functor $\extimes \from \aSet \times \rSet \to \arSet$ given by $(K \extimes L)_{a, r} = K_a \times L_r$.

\begin{definition} \label{def:diag-fun}
    A functor $\diag \from \arSet \to \aSet$ is \emph{(left) diagonal} if
    \begin{itemize}
        \item $\diag$ preserves colimits;
        \item $\diag$ preserves monomorphisms; and
        \item for any $K \in \rSet$, the composite 
        \[ \aSet \xrightarrow{\extimes K} \arSet \xrightarrow{\diag} \aSet \]
        preserves weak equivalences.
    \end{itemize}
\end{definition}
There is a notion of \emph{right diagonal functor} for functors $\raSet \to \aSet$ (where $\raSet$ denotes the category of presheaves on $R \times A$).
A functor $\diag \from \raSet \to \aSet$ is right diagonal if and only if the composite
\[ \arSet \xrightarrow{\simeq} \raSet \xrightarrow{\diag} \aSet \]
is left diagonal.
Thus, pre-composition by the equivalence $\arSet \simeq \raSet$ gives a bijection between right diagonal functors on $\raSet$ and left diagonal functors on $\arSet$ (ignoring size issues).
All statements that we make for left diagonal functors (in particular, the Generalized Diagonal Lemma \ref{thm:gen-diag-lem}) will immediately have analogues for right diagonal functors, hence we treat only the case of left diagonal functors and refer to them as simply \emph{diagonal functors}.
\begin{remark}
    If $\diag \from \arSet \to \aSet$ is a diagonal functor then, for any $K \in \rSet$, the functor 
    \[ \diag(- \extimes K) \from \aSet \to \aSet \] 
    is an exact functor between cofibration categories \cite[Def.~1.2]{szumilo}, though this will not play a role in the remainder of the paper.
\end{remark}

Given a diagonal functor $\diag \colon \arSet \to \aSet$, we write $\otimes$ for the composite $\diag \circ \extimes$.
The choice of notation here is meant to be suggestive, as many examples in practice occur in the case when $A = R$ and $\diag$ arises from a monoidal structure $\otimes$ on the category $\aSet$ (cf.~\cref{sec:examples} and \cref{ex:cisinski-monoidal-diag-weq}).

\begin{proposition} \label{birepr-compute}
    Let $\bi{bs} = (b, s)$ be an object in $A \times R$.
    For $a \in A$ and $r \in R$, we have an isomorphism
    \[ (L_r \repr{\bi{bs}})_a \cong A(a, b) \times \{ f \from r \to s \mid f_{-} \neq \id \}, \]
    natural in $\bi{bs}$.
\end{proposition}
Note the condition $f_- \neq \id$ is equivalent to the condition that $f \not\in R_+$.

\begin{proof}
    We first compute
    \[ \begin{array}{r@{\; }l@{\qquad}l}
        (L_r \repr{\bi{bs}})_a & \cong \colim\limits_{g \in \bd(r \slice R_{-})} (\repr{\bi{bs}}_{\cod(g)})_a \\
        & = \colim\limits_{g \in \bd(r \slice R_{-})} \big( A(a, b) \times R(\cod(g), s) \big) \\
        & \cong A(a, b) \times \colim\limits_{g \in \bd(r \slice R_{-})} R(\cod(g), s).
    \end{array} \]
    It remains only to construct a bijection
    \[ \colim\limits_{g \in \bd(r \slice R_{-})} R(\cod(g), s) \cong \{ f \from r \to s \mid f_- \neq \id \} \]
    natural in $s$ (since the above computation is natural in $b$).

    The colimit in question admits an explicit description
    \[ \colim\limits_{g \in \bd(r \slice R_{-})} R(\cod(g), s) \cong \Bigg\{ (h, g) \Bigg| \begin{array}{l}
        g \from r \to r', \\
        g \in R_{-} \setminus \{ \id \} , \\
        h \from r' \to s 
    \end{array} \Bigg\} \Bigg/ \sim, \]
    where $\sim$ is generated by identifications $(kh, g) \sim (k, hg)$ for any $k \in R$ and $h, g \in R_-$ such that $g \neq \id$ and the pairs $kh$ and $hg$ are composable.
    Let $S$ denote this set.

    Define a function $\Phi \from S \to \{ f \from r \to s \mid f_- \neq \id \}$ by
    \[ \Phi(h, g) := hg. \]
    Note that $(hg)_- \neq \id$ since $g$ strictly lowers the degree, hence $\Phi$ takes values in the codomain subset.
    This map is well-defined by associativity of composition.

    We claim this map is a bijection, with inverse given by
    \[ \Phi^{-1}(f) := (f_+, f_-). \]
    The inverse takes values in the set $S$ (i.e.\ $f_- \in R_- \setminus \{ \id \}$) by assumption.
    The equality $\Phi \circ \Phi^{-1} = \id$ is clear.
    To show that $\Phi^{-1} \circ \Phi = \id$, fix a pair $(h, g) \in S$ and factor $h$ as $h = h_+ h_-$.
    Since $g \in R_-$, it follows that $(hg)_- = h_- g$ and $(hg)_+ = h_+$.
    With this, we compute
    \begin{align*}
        \Phi^{-1}(\Phi(h, g)) &= \Phi^{-1}(hg) \\
        &= ((hg)_+, (hg)_-) \\
        &= (h_+, h_- g) \\
        &\sim (h_+ h_-, g) \\
        &= (h, g).
    \end{align*}

    Regarding naturality in $s$, both the domain and codomain set of $\Phi$ form functors in the variable $s$ by post-composition. 
    From this, it is immediate that the naturality squares commute.
\end{proof}

\begin{lemma} \label{ez-lemma}
    For $X \in \arSet$ and $n \geq 0$, the square
    \[ \begin{tikzcd}
        \coprod\limits_{\substack{r \in R \\ \deg (r) = n}} L_r X \extimes \repr{r} \push{L_r X \extimes \bd \repr{r}} X_r \extimes \bd \repr{r} \ar[r] \ar[d, shorten <= -1.5em] & \Sk^{n-1} X \ar[d] \\
        \coprod\limits_{\substack{r \in R \\ \deg(r) = n}} X_r \extimes \repr{r} \ar[r] & \Sk^{n} X
    \end{tikzcd} \]
    is a pushout.
\end{lemma}
\begin{proof}
    As every presheaf is a colimit of representable presheaves and colimits commute with colimits (in particular, $L_r$ commutes with colimits), it suffices to assume $X$ is a representable presheaf over some $\bi{bs} \in A \times R$.
    Instantiating this diagram at $\bi{at} \in A \times R$, it suffices to show the diagram
    \[ \begin{tikzcd}
        \coprod\limits_{\substack{r \in R \\ \deg (r) = n}} (L_r \repr{\bi{bs}} \extimes \repr{r} \push{L_r \repr{\bi{bs}} \extimes \bd \repr{r}} \repr{\bi{bs}}_r \extimes \bd \repr{r})_{a,t} \ar[r] \ar[d, hook, shorten <= -1.3em] & (\Sk^{n-1} \repr{\bi{bs}})_{a,t} \ar[d, hook] \\
        \coprod\limits_{\substack{r \in R \\ \deg(r) = n}} (\repr {\bi{bs}}_r \extimes \repr{r})_{a,t} \ar[r] & (\Sk^{n} \repr{\bi{bs}})_{a,t}
    \end{tikzcd} \]
    is a pushout of sets.

    For $r \in R$ such that $\deg(r) = n$, the top left set in the square may be explicitly described as
    \[ \coprod\limits_{\substack{r \in R \\ \deg(r) = n}} A(a, b) \times \{ (g \from r \to s, h \from t \to r) \mid g_{-} \neq \id \text{ or } h_{+} \neq \id \}, \]
    since
    \begin{align*}
        (L_r \repr{\bi{bs}} \extimes \repr{r})_{a,t} &= (L_r \repr{\bi{bs}})_a \extimes (\repr{r})_t \\
        &\cong A(a, b) \times \{ g \from r \to s \mid g_{-} \neq \id \} \times A(t, r)
    \end{align*}
    by \cref{birepr-compute}, and
    \begin{align*} 
        (\repr{\bi{bs}}_r \extimes \bd \repr{r})_{a,t} &= (\repr{\bi{bs}}_r)_a \times (\bd \repr{r})_d \\
        &= A(a, b) \times R(r, s) \times \{ h \from t \to r \mid h_{+} \neq \id \}.
    \end{align*}
    by \cref{repr-skeleton-description}.
    The top map
    in the square sends a tuple $(r, f, g, h)$ to the pair $(f, gh) \in \repr{\bi{bs}}_{a,t}$, which is an element of $(\Sk^{n-1} \repr{\bi{bs}})_{a,t}$ since $gh$ factors through some $r' \in R$ such that $\deg(r') < n$ (since either $g$ or $h$ factors).
    The bottom left set may be written as
    \[ \coprod\limits_{\substack{r \in R \\ \deg(r) = n}} (\repr{\bi{bs}}_r \extimes \repr{r})_{a,t} = \coprod\limits_{\substack{r \in R \\ \deg(r) = n}} A(a, b) \times R(r, s) \times R(t, r) \]
    and the bottom map sends a tuple $(r, f, g, h)$ to the pair $(f, gh) \in (\Sk^n \repr{\bi{bs}})_{a,t}$.
    
    Showing this square is a pushout, fix a set $S$ and a commutative square
    \[ \begin{tikzcd}
        \coprod\limits_{\substack{r \in R \\ \deg (r) = n}} (L_r \repr{\bi{bs}} \extimes \repr{r} \push{L_r \repr{\bi{bs}} \extimes \bd \repr{r}} \repr{\bi{bs}}_r \extimes \bd \repr{r})_{a,t} \ar[r] \ar[d, hook, shorten <= -1.35em] & (\Sk^{n-1} \repr{\bi{bs}})_{a,t} \ar[d, hook, "\Phi"] \\
        \coprod\limits_{\substack{r \in R \\ \deg(r) = n}} (\repr {\bi{bs}}_r \extimes \repr{r})_{a,t} \ar[r, "\Psi"] & S
    \end{tikzcd} \]
    Define a map $\Omega \from (\Sk^n \repr{\bi{bs}})_{a, t} \to S$ as follows: for $(f \from a \to b, \varphi \from t \to s) \in \Sk^n \repr{\bi{bs}}_{a, t}$, factor $\varphi$ as in the diagram
    \[ \begin{tikzcd}
        t \ar[rr, "\varphi"] \ar[rd, "\varphi_-"'] & {} & s \\
        {} & r \ar[ru, "\varphi_+"'] & {}
    \end{tikzcd} \]
    By \cref{repr-skeleton-description}, we have that $\deg(r) \leq n$.
    If $\deg(r) \neq n$ then $(f, \varphi)$ is in the image of the inclusion $(\Sk^{n-1} \repr{\bi{bs}})_{a, t} \ito (\Sk^n \repr{\bi{bs}})_{a, t}$.
    In this case, we define $\Omega$ by
    \[ \Omega(f, \varphi) := \Phi(f, \varphi). \]
    Otherwise, we have that $\deg(r) = n$, hence $(r, f, \varphi_+, \varphi_-)$ is an element of the bottom left set.
    Thus, we may define $\Omega$ by
    \[ \Omega(f, \varphi) := \Psi(r, f, \varphi_+, \varphi_-). \]

    It remains to show $\Omega$ is the unique map making the diagram
    \[ \begin{tikzcd}
        \coprod\limits_{\substack{r \in R \\ \deg(r) = n}} (\repr{\bi{bs}}_r \extimes \repr{r})_{a, t} \ar[r] \ar[rd, "\Psi"{pos=0.2}, shorten <= -2.8em] & (\Sk^n \bi{\repr{bs}})_{a, t} \ar[d, "\Omega"{pos=0.46, description}, dotted] & (\Sk^{n-1} \bi{\repr{bs}})_{a, t} \ar[l, hook'] \ar[ld, "\Phi"'{pos=0.52}] \\
        {} & S & {}
    \end{tikzcd} \]
    commute.    
    Commutativity of the right triangle follows by construction.
    For the left triangle, fix a tuple $(r, f, g, h)$.
    If either $g_{-} \neq \id$ or $h_+ \neq \id$ then commutativity follows from commmutativity of the starting square and the right triangle.
    Otherwise, it must be that $g \in R_{+}$ and $h \in R_{-}$. 
    By uniqueness of factorizations, we have that $(gh)_+ = g$ and $(gh)_- = h$, thus $\Omega(f, gh) = \Psi(r, f, g, h)$ as desired.
    If $\Omega' \from (\Sk^n \repr{\bi{bs}})_{a, t} \to S$ also makes the diagram commute then, for $(f, \varphi) \in (\Sk^n \repr{\bi{bs}})_{a, t}$, we factor $\varphi$ through some object $r \in R$ as $\varphi = \varphi_+ \varphi_-$ as before.
    If $\deg(r) < n$ then
    \[ \Omega(f, \varphi) = \Phi(f, \varphi) = \Omega'(f, \varphi). \]
    Otherwise, if $\deg(r) = n$ then
    \[ \Omega(f, \varphi) = \Psi(r, f, \varphi_+, \varphi_-) = \Omega'(f, \varphi_+ \varphi_-) = \Omega'(f, \varphi). \qedhere \]
\end{proof}

The following two lemmas (alongside \cref{reedy-cofib-colim-weq}) encapsulate the role of Reedy cofibrancy in the proof of the Generalized Diagonal Lemma.
\begin{lemma} \label{reedy-cofib-po-map-is-cofib}
    If $X \from R^\op \to \aSet$ is Reedy cofibrant then
        for $n \geq 0$ and $r \in R$ with $\deg(r) = n$, the maps
        \[ L_r X \extimes \repr{r} \cup_{L_r X \extimes \bd \repr{r}} X_r \extimes \bd \repr{r} \to X_r \extimes \repr{r} \]
        and
        \[ \Sk^{n-1} X \to \Sk^{n} X \]
        are monomorphisms.
\end{lemma}
\begin{proof}
    For the first map, we instantiate at $\bi{at} \in A \times R$ and show that
    \[ (L_r X \extimes \repr{r} \cup_{L_r X \extimes \bd \repr{r}} X_r \extimes \bd \repr{r})_{a, t} \to (X_r \extimes \repr{r})_{a, t} \]
    is injective.
    Using that pushouts of presheaves commute with evaluation, this map is induced via universal property from the commutative square
    \[ \begin{tikzcd}
        (L_r X)_a \times (\bd \repr{r})_{t} \ar[r] \ar[d] & (L_r X)_a \times (\repr{r})_{t} \ar[d] \\
        X_{a, r} \times (\bd \repr{r})_{t} \ar[r] & X_{a, r} \times (\repr{r})_{t} 
    \end{tikzcd} \]
    The map $(L_r X)_a \to X_{a, r}$ is injective since $X$ is Reedy cofibrant (\cref{elegant-reedy-cofibrant}).
    The map $\bd \repr{r} \to \repr{r}$ is injective by \cref{repr-skeleton-description}.
    The desired map is the pushout-product of two injections, hence is injective.

    The second map is a monomorphism by \cref{ez-lemma}, since it is a pushout (in $\arSet$) of the first map.
\end{proof}
\begin{lemma} \label{reedy-cofib-latching-diagram}
    If $X \from R^\op \to \aSet$ is Reedy cofibrant then
    the diagram
    \[ \bd (r \slice R_-)^\op \to R^\op \xrightarrow{X} \aSet \]
    is Reedy cofibrant.
\end{lemma}
\begin{proof}
    The latching category of an object $\sigma \from r \to s$ in $\bd(r \slice R_-)$ is isomorphic to the latching category of $s \in R$ by \cite[pg.~102, before Def.~9.1.3]{radulescu-banu}.
    Since $X$ is Reedy cofbrant, the map $L_s X \to X_s$ is a monomorphism for any $s$, therefore the map from the latching category of $\sigma \from r \to s$ to $X_{\cod \sigma}$ is a monomorphism.
\end{proof}

We may now prove the Generalized Diagonal Lemma.
\begin{theorem}[Generalized Diagonal Lemma] \label{thm:gen-diag-lem}
    Let $f \from X \to Y$ be a morphism in $\arSet$ between Reedy cofibrant diagrams such that $f_r \from X_r \to Y_r$ is a weak equivalence in $\aSet$ for all $r \in R$.
    Then, $\diag f \from \diag X \to \diag Y$ is a weak equivalence.
\end{theorem}
\begin{proof}
    The maps $\Sk^n X \to \Sk^{n+1} X$ and $\Sk^n Y \to \Sk^{n+1} Y$ are monomorphisms by \cref{reedy-cofib-po-map-is-cofib}.
    As $\diag$ preserves colimits and monomorphisms, the diagrams
    \[ \diag \Sk^{-1} X \to \diag \Sk^0 X \to \diag \Sk^1 X \to \ldots \to \diag X \]
    and
    \[ \diag \Sk^{-1} Y \to \diag \Sk^0 Y \to \diag \Sk^1 Y \to \ldots \to \diag Y \]
    are colimit diagrams valued in monomorphisms.
    Any diagram $\mathbb{N} \to \aSet$ taking values in monomorphisms is Reedy cofibrant, so by \cref{reedy-cofib-colim-weq},
    it suffices to show $\diag \Sk^n f \from \diag \Sk^n X \to \diag \Sk^n Y$ is a weak equivalence for $n \geq -1$.
    For $n = -1$, this is immediate. 
    For $n = 0$, this follows by assumption.

    By induction, fix $n > 0$ and suppose $\diag \Sk^{n-1} f \from \diag \Sk^{n-1} X \to \diag \Sk^{n-1} Y$ is a weak equivalence.
    By \cref{ez-lemma}, the front and back squares in
    \[ \begin{tikzcd}[cramped, sep = tiny]
        \coprod\limits_{\substack{r \in R \\ \deg(r) = n}} L_r X \extimes \repr{r} \push{L_r X \extimes \bd \repr{r}} X_r \extimes \bd \repr{r} \ar[rr] \ar[dd, hook, shorten <= -1.5em] \ar[rd, shorten <= -3em, shorten >= 0.5em] & {} & \Sk^{n-1} X \ar[dd, hook] \ar[rd] & {} \\
        {} & \coprod\limits_{\substack{r \in R \\ \deg(r) = n}} L_r Y \extimes \repr{r} \push{L_r Y \extimes \bd \repr{r}} Y_r \extimes \bd \repr{r} \ar[rr, crossing over] & {} & \Sk^{n-1} Y \ar[dd, hook] \\
        \coprod\limits_{\substack{r \in R \\ \deg(r) = n}} X_r \extimes \repr{r} \ar[rr] \ar[rd, shorten >= -1em] & {} & \Sk^n X \ar[rd] & {} \\
        {} & \coprod\limits_{\substack{r \in R \\ \deg(r) = n}} Y_r \extimes \repr{r} \ar[rr] \ar[from=uu, hook, crossing over, shorten <= -1.5em] & {} & \Sk^n Y 
    \end{tikzcd} \]
    are pushouts.
    Applying $\diag$, the front and back squares in
    \[ \begin{tikzcd}[cramped, sep = tiny]
        \coprod\limits_{\substack{r \in R \\ \deg(r) = n}} L_r X \otimes \repr{r} \push{L_r X \otimes \bd \repr{r}} X_r \otimes \bd \repr{r} \ar[rr] \ar[dd, hook, shorten <= -1.5em] \ar[rd, shorten <= -3em, shorten >= 0.5em] & {} & \diag \Sk^{n-1} X \ar[dd, hook] \ar[rd] & {} \\
        {} & \coprod\limits_{\substack{r \in R \\ \deg(r) = n}} L_r Y \otimes \repr{r} \push{L_r Y \otimes \bd \repr{r}} Y_r \otimes \bd \repr{r} \ar[rr, crossing over] & {} & \diag \Sk^{n-1} Y \ar[dd, hook] \\
        \coprod\limits_{\substack{r \in R \\ \deg(r) = n}} X_r \otimes \repr{r} \ar[rr] \ar[rd, shorten >= -1em] & {} & \diag \Sk^n X \ar[rd] & {} \\
        {} & \coprod\limits_{\substack{r \in R \\ \deg(r) = n}} Y_r \otimes \repr{r} \ar[rr] \ar[from=uu, hook, crossing over, shorten <= -1.5em] & {} & \diag \Sk^n Y 
    \end{tikzcd} \]
    are again pushouts as $\diag$ is cocontinuous.
    The map between the top right objects is a weak equivalence by the inductive hypothesis.
    The map between the bottom left objects is a weak equivalence since a coproduct of weak equivalences is a weak equivalence \cite[Lem.~1.6.3.(1)]{radulescu-banu}.
    The left maps in both the front and back squares are cofibrations (\cref{reedy-cofib-po-map-is-cofib}), so by the gluing lemma \cite[Lem.~1.4.1.(1b)]{radulescu-banu}, it suffices to show the map between the top left objects is a weak equivalence.
    This map is a coproduct of maps between pushouts which appear in the bottom right of
    \[ \begin{tikzcd}
        L_r X \otimes \bd \repr{r} \ar[rr] \ar[dd, hook] \ar[rd] & {} & X_r \otimes \bd \repr{r} \ar[dd, hook] \ar[rd] & {} \\
        {} & L_r Y \otimes \bd \repr{r} \ar[rr, crossing over] & {} & Y_r \otimes \bd \repr{r} \ar[dd, hook] \\
        L_r X \otimes \repr{r} \ar[rr] \ar[rd] & {} & L_r X \otimes \repr{r} \push{L_r X \otimes \bd \repr{r}} X_r \otimes \bd \repr{r} \ar[rd] & {} \\
        {} & L_r Y \otimes \repr{r} \ar[rr] \ar[from=uu, crossing over] & {} & L_r Y \otimes \repr{r} \push{L_r Y \otimes \bd \repr{r}} Y_r \otimes \bd \repr{r}
    \end{tikzcd} \]
    The map between the top right objects is a weak equivalence by assumption.
    Both the top and left maps in the front and back squares are cofibrations (by Reedy cofibrancy and by \cref{repr-skeleton-description}, respectively).
    By the gluing lemma, it suffices to show $L_r X \to L_r Y$ is a weak equivalence.


    The map $f$ induces a pointwise weak equivalence between diagrams
    \[ \begin{tikzcd}
    \bd(r \slice R_{-})^\op  
       \ar[r]
    &
    R^\op
      \ar[r, yshift = .75ex, "X"]
      \ar[r, yshift = -.75ex, "Y", swap]
    &
    \aSet .
    \end{tikzcd}
    \]
    These diagrams are Reedy cofibrant by \cref{reedy-cofib-latching-diagram}, hence by \cref{reedy-cofib-colim-weq}, the induced map between colimits $L_r X \to L_r Y$ is a weak equivalence.
\end{proof}
Using \cref{elegant-reedy-cofibrant}, we can rephrase the assumptions to obtain the following corollary.
\begin{corollary} \label{cor:}
    Suppose $R$ is an elegant Reedy category.
    Let $f \from X \to Y$ be a morphism in $\arSet$ such that $f_r \from X_r \to Y_r$ is a weak equivalence in $\aSet$ for all $r \in R$.
    Then, $\diag f \from \diag X \to \diag Y$ is a weak equivalence. \qed
\end{corollary}

  \section{Examples} \label{sec:examples}



We now give several examples of applications of \cref{thm:gen-diag-lem}.
Throughout this section, we still follow \cref{A-ez-aset-Cisinski}: $A$ and $R$ are Reedy categories, and $\aSet$ is considered with a Cisinski cofibration category structure.
Furthermore, in the examples presented below, $R$ will always be an elegant Reedy category, hence all $R$-presheaves are Reedy cofibrant.
 
If $\aSet$ carries a bicocontinuous (right) $\rSet$-action $\otimes \from \aSet \times \rSet \to \aSet$ then one might define a functor $\diag \from \arSet \to \aSet$ as a left Kan extension
\[
  \begin{tikzcd}[column sep=huge]
    A \times R 
      \ar[r, "\otimes"]
      \ar[d, hook]
    & \aSet 
    \\
    \arSet
      \ar[ru, bend right, "\diag", swap, near end]
    &
  \end{tikzcd}
\]

It is easy to see that the composite $\diag \circ \extimes$ defines an $\rSet$-action on $\aSet$ that agrees with the original action.
Analogously, if $\aSet$ carries a bicocontinuous left $\rSet$-action then we may define $\diag$ in a similar way to obtain a right diagonal functor.


\begin{corollary}
  Suppose $R$ is elegant and $\aSet$ is equipped with a (left or right) $\rSet$-action which preserves colimits in each variable, and weak equivalences in the $\aSet$ variable.

  If $f \colon X \to Y$ is a map in $\arSet$ such that $f_r \colon X_r \to Y_r$ is a weak equivalence for every $r \in R$, then $\diag f \colon \diag X \to \diag Y$ is a weak equivalence.
\end{corollary}

\begin{proof}
  The functor $\diag$ is a diagonal functor in the sense of \cref{def:diag-fun}, and hence we may apply \cref{thm:gen-diag-lem}.
\end{proof}

\begin{example} \label{ex:simplicial-objects-in-rSet}
  If $\aSet$ is a simplicial model category whose cofibrations are the monomorphisms then the tensoring of weak equivalences is a weak equivalence (by the ``pushout-product axiom'').
  As $\Delta$ is elegant, any levelwise weak equivalence $f_\bullet \from X_\bullet \to Y_\bullet$ between simplicial objects in $\aSet$ induces a weak equivalence $\diag f_\bullet \from \diag X_\bullet \to \diag Y_\bullet$.
\end{example}
Reversing the role of the simplex category in \cref{ex:simplicial-objects-in-rSet} yields another example.
\begin{example}
  Both the Kan--Quillen and Joyal model structures on $\sSet$ are Cisinski model structures.
  Thus, if $R$ is elegant and $\sSet$ has a bicocontinuous tensoring $\otimes$ by $\rSet$ satisfying the pushout-product axiom then any levelwise weak equivalence $f \from X \to Y$ between $R$-diagrams valued in $\sSet$ induces a weak equivalence $\diag f \from \diag X \to \diag Y$ in $\sSet$.
\end{example}

When $A = R$, an $\rSet$ action on $\aSet$ is exactly a monoidal product on $\aSet$.
Many natural applications of the Generalized Diagonal Lemma arise in this way.

\begin{example} \label{ex:cisinski-monoidal-diag-weq}
  If $\aSet$ is a Cisinski monoidal model category (i.e., a monoidal model category whose cofibrations are the monomorphisms), then the monoidal product of weak equivalences is again a weak equivalence.
  Thus in any such case a levelwise weak equivalence $f \from X \to Y$ induces a weak equivalence $\diag f \colon \diag X \to \diag Y$.
\end{example}

\begin{example}[Geometric product of cubical sets]
  Consider the box category $\Box$ with zero, one, or two connections, but no symmetries, reversals, or diagonals.
  This category carries a monoidal structure given by $([1]^m, [1]^n) \mapsto [1]^{m+n}$, giving rise to the functor $\diag \colon \ccSet \to \cSet$ whose composite $\otimes := \diag \circ \extimes$ is the \emph{geometric product} on cubical sets.
  Since cubical sets form a monoidal model category (the Grothendieck model structure with the geometric product), the product of weak equivalences is again a weak equivalence.
  Hence any map $f \colon X \to Y$ of bicubical sets that is a levelwise weak equivalence induces a weak equivalence $\diag f \colon \diag X \to \diag Y$.
\end{example}

\begin{example}[Join of simplicial sets]
  Consider the promonoidal structure $\Delta \times \Delta \to \sSet$ on the simplex category given by $([m], [n]) \mapsto \simp{m+n+1}$.
  The resulting diagonal functor $\diag[\join] \colon \ssSet \to \sSet$ composed with the external product yields the \emph{join} structure on simplicial sets, which preserves both weak homotopy equivalences and weak categorical equivalences.
  Thus any map $f \colon X \to Y$ of bisimplicial sets that is a levelwise weak equivalence induces a weak equivalence $\diag[\join] f \colon \diag[\join] X \to \diag[\join] Y$.
\end{example}

Another class of promonoidal structures $A \times A \to \aSet$ comes from the categorical product via the functor $(a, b) \mapsto \repr{a} \times \repr{b}$.
Put differently, given an elegant Reedy category $A$, we have the canonical categorical diagonal inclusion $(\id, \id) \colon A \to A \times A$ sending $a$ to $(a, a)$, which induces an adjoint triple
\[
  \begin{tikzcd}[column sep=huge]
    \aaSet 
      \ar[r, bend left]
      \ar[r, bend right]
    & \aSet 
      \ar[l,"\diagonal" description]
  \end{tikzcd}
\]
where the middle functor $\diagonal \colon \aaSet \to \aSet$ is given by precomposition with the inclusion $A \to A \times A$.

\begin{corollary} \label{thm:diag-lem-cat-prod}
  Suppose the product $w \times w'$ of weak equivalences in $\aSet$ is a weak equivalence.
 
  If $f \colon X \to Y$ is a map in $\aaSet$ such that $f_a \colon X_a \to Y_a$ is a weak equivalence for every $a \in A$, then $\diagonal f \colon \diagonal X \to \diagonal Y$ is a weak equivalence. \qed
\end{corollary}

\begin{example}[Joyal model structure on simplicial sets]
  The Joyal model structure on simplicial sets is monoidal with respect to the categorical product.
  Hence if $f \colon X \to Y$ is a bisimplicial map such that $f_n \colon X_n \to Y_n$ is a weak categorical equivalence for every $n \in \mathbb{N}$, then $\diagonal f \colon \diagonal X \to \diagonal Y$ is a weak categorical equivalence.
\end{example}

For instance, if $A$ is a strict test category \cite[Def.~1.6.7]{maltsiniotis:book}, then the weak equivalences of $\aSet$ are closed under finite products.
This gives:

\begin{corollary}
  Let $A$ be an elegant Reedy category that is also a strict test category and let $\aSet$ be equipped with the corresponding canonical model structure.

  If $f \colon X \to Y$ is a map in $\aaSet$ such that $f_a \colon X_a \to Y_a$ is a weak equivalence for every $a \in A$, then $\diagonal f \colon \diagonal X \to \diagonal Y$ is a weak equivalence. \qed
\end{corollary}

\begin{example}[Kan--Quillen model structure on simplicial sets]
  The simplex category $\Delta$ is a strict test category \cite[Prop.~1.6.14]{maltsiniotis:book} and the test category model structure coincides with the Kan--Quillen model structure thereon, which allows us to recover the usual Diagonal Lemma.
  (Of course, there are many simpler ways of showing that weak homotopy equivalences of simplicial sets are closed under products.)
\end{example}

\begin{example}[Cubical sets]
  The box category $\Box$ with one or two connections (but again, no symmetries, reversals, or diagonals) is a strict test category \cite[Prop.~4.3]{maltsiniotis:cubes}, and hence any map of bicubical sets $f \colon X \to Y$ that is a levelwise equivalence induces a weak equivalence $\diagonal f \colon \diagonal X \to \diagonal Y$.
  
  Note however that this would not be true in the minimal box category, i.e., without connections.
  Since the categorical product $\Box^1 \times \Box^1$ has the homotopy type of $S^2 \vee S^1$ (cf. \cite[\S6]{maltsiniotis:cubes}), the product of the weak equivalence $\Box^1 \to \Box^0$ with itself is not a weak equivalence.
\end{example}



\renewcommand{\thesection}{\Alph{section}}



  

\bibliographystyle{amsalphaurlmod}
\bibliography{all-refs.bib}

\end{document}